\documentclass{amsart}
\usepackage{graphicx,amsfonts,amssymb,amsmath,amsthm}
\usepackage[pdftex]{hyperref}  
\usepackage{cite}  

\theoremstyle{plain} 
\newtheorem{theorem}    {Theorem}[section] 
\newtheorem{lemma}      [theorem]{Lemma}

\theoremstyle{definition}

\theoremstyle{remark}
\newtheorem{remark}              {Remark}

\numberwithin{equation}{section}

\def\C{\mathbb C}

\def\Q{\mathbb Q}
\def\R{\mathbb R}

\begin{document}

\title{On the occurrence of large positive Hecke eigenvalues for {\rm GL}(2)}

\author{Nahid Walji}
\address{The American University of Paris, 102 rue Saint-Dominique, Paris, France}
\email{nwalji@aup.edu}
\maketitle
\begin{abstract} Let $\pi$ be a self-dual cuspidal automorphic representation for {\rm GL}(2)/$\Q$. We show that there exists a positive upper Dirichlet density of primes at which the associated Hecke eigenvalues of $\pi$ are larger than a specified positive constant.
\end{abstract}

\section{Introduction}

Let $\pi$ be a cuspidal automorphic representation for {\rm GL}(2)/$\Q$ that is self-dual. To each prime $p$ at which $\pi$ is not ramified is associated a Hecke eigenvalue, denoted by $a_p = a_p(\pi)$. The values taken by sequences $(a_p(\pi))_p$ of  Hecke eigenvalues have been long-studied from various points of view. 
In 1994, J.-P. Serre asked (see appendix of~\cite{Sh94}) whether it is possible to find positive constants $c, c'$ such that, for all $\epsilon > 0$, there exist infinitely many $a_p$ greater than $c -\epsilon$ and infinitely many $a_p$ less than $-c' + \epsilon$. He then proved such results for the case of (certain) modular forms, and asked if similar results can be shown to hold in the case of Maass forms.

In ~\cite{Wa16} we proved, for any self-dual cuspidal automorphic representation $\pi$ for {\rm GL}(2)/$\Q$, that for any positive $\epsilon$ there exist infinitely many primes p such that $$a_p > 0.905... - \epsilon$$  and if $\pi$ is non-dihedral, then there exist infinitely many primes $p$ such that $$a_p < -1.164... + \epsilon$$ (precise expressions for the constants are available in ~\cite{Wa16}). Note that a cuspidal automorphic representation $\pi$ for {\rm GL}(2) is said to be \textit{dihedral} if it is associated to a 2-dimensional irreducible Artin representation $\rho$ that is of dihedral type, meaning that the image of $\rho$ in ${\rm PGL}_2(\C)$ is isomorphic to a dihedral group.
Furthermore, $\pi$ is said to be of \textit{solvable polyhedral type} if it is associated to a 2-dimensional irreducible Artin representation $\rho$ that is of dihedral, tetrahedral, or octahedral type (which means that the projective image of $\rho$ in ${\rm PGL}_2(\C)$ is isomorphic to a dihedral group, $A_4$, or $S_4$, respectively).

A related question is to ask whether it is possible to obtain similar statements for not just an infinitude of primes but a positive upper Dirichlet density of primes. Recall that the \textit{upper Dirichlet density} of a set $S$ of primes is defined to be
\begin{align*}
\overline{\delta}(S) := \lim_{s \rightarrow 1 ^+}\sup \frac{\sum_{p \in S}p ^{-s}}{\log (1/ (s-1))}.
\end{align*}
The results of ~\cite{Wa16} relied on determining lower bounds on the asymptotic growth of certain Dirichlet series, but, because of the lack of knowledge of the Ramanujan conjecture (or indeed the non-existence of any known uniform bound on the Satake parameters), these will not directly lead to a positive density result.

In this paper, we outline a method to circumvent this issue and obtain positive upper Dirichlet density results. By way of example, we have 

\begin{theorem}\label{t}
Let $\pi$ be a self-dual cuspidal automorphic representation for ${\rm GL}(2)$ over $\Q$ that is not of solvable polyhedral type. Then for any $\epsilon > 0$, the set
\begin{align*}
\{p \mid a_p (\pi) > 0.778... - \epsilon \}.
\end{align*}
has an upper Dirichlet density of at least $1/100$.
\end{theorem}

The exact value of the constant in the set condition is $0.36729 ^{1/4}$, which is determined in Section ~\ref{proof}. The method we use would also allow a change in this constant to a smaller value so as to obtain a mild increase in the lower bound of the density.
The proof relies in part on the deep work of Gelbart--Jacquet ~\cite{GJ78}, Kim--Shahidi ~\cite{KS00,KS02}, and Kim ~\cite{Ki03} on the automorphy of symmetric power lifts. In the next section, we will outline the ingredients used in the proof, and in Section~\ref{proof} we will prove the theorem.

\section{Background} \label{bkgd}

The proof will rely on the study of the asymptotic behaviour of various Dirichlet series, which we will briefly outline here and refer the reader to~\cite{Wa16} for a more detailed explanation.

Given a cuspidal automorphic representation $\pi$ for {\rm GL}(2)/$\Q$, let $T$ be the (finite) set consisting of the archimedean place and the finite places at which $\pi$ is ramified. Then the incomplete $L$-function (with respect to $T$) that is associated to $\pi$ can be defined in a right-half plane via an Euler product: 
\begin{align*}
L^T(s,\pi) =\prod_{p \not \in T} {\rm det}\left(I_2 - A_p (\pi) p ^{-s}\right)^{-1},
\end{align*}
where $I_2$ is the $2 \times 2$ identity matrix and $A_p (\pi) = {\rm diag}(\alpha_p (\pi), \beta_p (\pi)) \in {\rm GL}_2(\C)$ is the matrix of Satake parameters associated to $\pi$ at $p$ (see Section 3.5 of~\cite{Bu97} for further background on Satake parameters).
Note that the Hecke eigenvalue $a_p(\pi)$ is equal to the sum of the Satake parameters $\alpha_p (\pi)$ and $\beta_p (\pi)$.

For any two cuspidal automorphic representations $\pi_1$ for {\rm GL}(n)/$\Q$ and $\pi_2$ for {\rm GL}(m)/$\Q$, one can define (again in a suitable right-half plane) their incomplete Rankin--Selberg $L$-function: 
\begin{align*}
L^T(s, \pi_1 \times \pi_2) =\prod_{p \not \in T} {\rm det}\left(I_{nm} - A_p (\pi_1) \otimes A_p (\pi_2) p ^{-s}\right)^{-1}.
\end{align*}
This $L$-function converges absolutely for ${\rm Re}(s)>1$. At $s=1$ it has a simple pole iff $\pi_1$ is dual to $\pi_2$~\cite{JS81}, otherwise the $L$-function is invertible at that point~\cite{Sh81}.

In general one can define, for any positive integer $k \leq 8$, an incomplete $k$th product $L$-function as follows: 
\begin{align*}
L^T(s,\pi ^{\times k}) =\prod_{p \not \in T} {\rm det}\left(I_{2^k} - A_p (\pi)^{\otimes k} p ^{-s}\right)^{-1}.
\end{align*}
One can also define the following incomplete symmetric power $L$-functions:
\begin{align*}
L^T(s,{\rm Sym}^2 \pi) &=\prod_{p \not \in T} {\rm det}\left(I_{3} -      \left( \begin{array}{ccc}
\alpha_p ^2&  &  \\
& \alpha_p \beta_p &  \\
&  & \beta_p ^2
\end{array} \right) p ^{-s}\right)^{-1},\\
L^T(s,{\rm Sym}^3 \pi) &=\prod_{p \not \in T} {\rm det}\left(I_{4} -      \left( \begin{array}{cccc}
\alpha_p ^3&  &  &\\
& \alpha_p ^2 \beta_p& &  \\
&  & \alpha_p \beta_p ^2& \\
&&& \beta_p ^3
\end{array} \right) p ^{-s}\right)^{-1},\\
L^T(s,{\rm Sym}^4 \pi) &=\prod_{p \not \in T} {\rm det}\left(I_{5}-\left( \begin{array}{ccccc}
\alpha_p ^4&&&  &  \\
& \alpha_p ^3 \beta_p &  &&\\
&& \alpha_p ^2 \beta_p ^2 && \\
&&& \alpha_p \beta_p ^3 & \\
&&&  & \beta_p ^4
\end{array} \right) p ^{-s}\right)^{-1}.
\end{align*}
For the $k$th product $L$-functions, where $k = 3,4,6,$ and $8$, we have the following $L$-function identities, using Clebsch--Gordon decompositions (for details of these decompositions for $L^T(s,\pi^{\times 3})$ and $L^T(s,\pi^{\times 4})$ when $\pi$ has trivial central character, see p74--75 and p70--71 (respectively) of \cite{Wal11} and note that the other cases follow in the same way):
\begin{align*}
L^T(s, \pi ^{\times 3}) =& L^T(s, {\rm Sym}^3 \pi) L^T(s, \pi \otimes \omega)^2 ,\\
L^T(s, \pi ^{\times 4}) =& L^T(s, {\rm Sym}^4 \pi) L^T(s, {\rm Sym}^2 \pi \otimes \omega)^3 L^T(s, \omega ^2)^2,\\
L^T(s, \pi ^{\times 6}) =& L^T(s, {\rm Sym}^3 \pi \times {\rm Sym}^3 \pi) 
L^T(s, {\rm Sym}^3 \pi \times \pi \otimes \omega )^4 
L^T(s, \pi \times \pi \otimes \omega ^2)^4,\\
L^T(s, \pi ^{\times 8}) =& L^T(s, {\rm Sym}^4 \pi \times {\rm Sym}^4 \pi) 
L^T(s, {\rm Sym}^4 \pi \times {\rm Sym}^2\pi \otimes \omega)^6 \\
\cdot & 
L^T(s, {\rm Sym}^2 \pi \otimes \omega \times {\rm Sym}^2\pi \otimes \omega)^9
L^T(s, {\rm Sym}^4 \pi \otimes \omega ^2)^4 \\
\cdot &
L^T(s, {\rm Sym}^2 \pi \otimes \omega ^3)^{12}
L^T(s, \omega ^4)^4,
\end{align*}
where $\omega$ is the central character of $\pi$.

From here on, we assume that $\pi$ is self-dual and that it is not of solvable polyhedral type. We use the equations above in conjunction with the results of Gelbart--Jacquet~\cite{GJ78}, Kim--Shahidi~\cite{KS00,KS02}, and Kim~\cite{Ki03} on the automorphy of the symmetric second, third, and fourth power lifts to obtain the following:\\
For $k = 3,4,6,$ and $8$, the incomplete $L$-function $L^T(s, \pi ^{\times k})$ has an absolutely convergent Euler product for $s > 1$. If $k$ is even, then the incomplete $L$-function has a pole of order $m(k)$ at $s=1$, where $m(k) = 2,5,14$ for $k = 4,6,8$, respectively. If $k=3$, then the $L$-function is invertible at $s=1$.\\

Using the bounds towards the Ramanujan conjecture obtained by Kim--Sarnak (Appendix 2 of~\cite{Ki03}), which imply that $|a_p|\leq 2p ^{7/64}$ for all primes $p$, we obtain the following results:

As $s \rightarrow 1^+$, for $k = 3$ or $4$, we have
\begin{align*}
\sum_{p \not \in T}\frac{a_p^k}{p^s} &= \log L^T(s, \pi ^{\times k}) + O\left(1\right) 
\end{align*}
and for $k = 6$ or $8$, using the positivity of the coefficients of $p^{-ns}$ in the expansion of $\log L^T(s, \pi ^{\times k})$, we have
\begin{align*}
\sum_{p \not \in T}\frac{a_p^k}{p^s} &\leq \log L^T(s, \pi ^{\times k}).
\end{align*}

Given the orders $m (k)$ of the poles of the (incomplete) product $L$-functions as determined above, we conclude:
\begin{align*}
\sum_{p \not \in T}\frac{a_p^k}{p^s} =
\left\{ \begin{array}{lccl}
     O\left(1\right) & \text{ for }&$k = 3$\\
     2\log (1 / (s-1)) + O\left(1\right)& \text{ for }&$k = 4$\\
     \end{array} \right.\\ 
\end{align*}
and     
\begin{align*}    
\sum_{p \not \in T}\frac{a_p^k}{p^s} \leq \left\{ \begin{array}{cccl}
     5\log (1 / (s-1)) + O\left(1\right)& \text{ for }&$k = 6$\\
     14\log (1 / (s-1)) + O\left(1\right)& \text{ for }&$k = 8$
     \end{array} \right.
\end{align*}
as $s \rightarrow 1^+$.\\

We also make use of the following identities:\\

Let $f(x), g(x)$ be real-valued functions, and fix some point $u \in \R$. Then,
\begin{align*}
\lim_{x \rightarrow u^+} {\rm sup} \left(f(x) + g(x)\right) &\leq \lim_{x \rightarrow u^+} {\rm sup}\ f(x) + \lim_{x \rightarrow u^+} {\rm sup}\ g(x) \\
\lim_{x \rightarrow u^+} {\rm sup} (-f(x)) &= - \lim_{x \rightarrow u^+} {\rm inf}\ f(x) 
\end{align*}
and furthermore if $g,f$ are non-negative functions, then 
\begin{align*}
\lim_{x \rightarrow u^+} {\rm inf} \left(f(x) \cdot g(x)\right) &\leq \lim_{x \rightarrow u^+} {\rm inf}\ f(x) \cdot \lim_{x \rightarrow u^+} {\rm sup}\ g(x) \\
\lim_{x \rightarrow u^+} {\rm sup} \left(f(x) \cdot g(x)\right) &\leq \lim_{x \rightarrow u^+} {\rm sup}\ f(x) \cdot \lim_{x \rightarrow u^+} {\rm sup}\ g(x) .
\end{align*}

\subsection{Proof}\label{proof}

In this subsection we will prove Theorem~\ref{t}.

First define the sets $A := \{p \text{ prime }\mid a_p > 0\}$ and $B:=\{p \text{ prime }\mid a_p \leq 0\}$.
From the previous section, we know that 
\begin{align*}
\lim_{s \rightarrow 1 ^+}{\rm sup} \frac{\sum_{p}\frac{|a_p|^4}{p^s}}{\log \left(\frac{1}{s-1}\right)} = 2,
\end{align*}
which implies 
\begin{align*}
\lim_{s \rightarrow 1 ^+}{\rm sup} \frac{\sum_{p \in A}\frac{|a_p|^4}{p^s}}{\log \left(\frac{1}{s-1}\right)} + \lim_{s \rightarrow 1 ^+}{\rm sup} \frac{\sum_{p \in B}\frac{|a_p|^4}{p^s}}{\log \left(\frac{1}{s-1}\right)} \geq 2.
\end{align*}
Let us define 
\begin{align*}
d := \lim_{s \rightarrow 1 ^+}{\rm sup} \frac{\sum_{p \in B}\frac{|a_p|^4}{p^s}}{\log \left(\frac{1}{s-1}\right)},
\end{align*}
and thus we can write
\begin{align*}
\lim_{s \rightarrow 1 ^+}{\rm sup} \frac{\sum_{p \in A}\frac{|a_p|^4}{p^s}}{\log \left(\frac{1}{s-1}\right)} \geq 2 - d.
\end{align*}
Now define $S_\beta \subset A$ to be exactly the set of primes $p$ such that $|a_p|^4 \geq (2-d)\beta$, where $0 < \beta < 1$ is a constant to be fixed later.
We will assume that $S_\beta$ has an upper Dirichlet density smaller than $1/100$. 

We have the bound 
\begin{align*}
\lim_{} {\rm sup} \left(\frac{\sum_{p \in A-S_\beta}\frac{|a_p|^4}{p^s}}{\log \left(\frac{1}{s-1}\right)}\right)
\leq  (2-d)\beta.
\end{align*}

Since
\begin{align*}
\lim_{} {\rm sup} \left(\frac{\sum_{p \in A-S_\beta}\frac{|a_p|^4}{p^s}}{\log \left(\frac{1}{s-1}\right)}\right) 
+ \lim_{} {\rm sup} \left(\frac{\sum_{p \in S_\beta}\frac{|a_p|^4}{p^s}}{\log \left(\frac{1}{s-1}\right)}\right)
\geq (2-d),
\end{align*}
we have 
\begin{align*}
\lim_{} {\rm sup} \left(\frac{\sum_{p \in S_\beta}\frac{|a_p|^4}{p^s}}{\log \left(\frac{1}{s-1}\right)}\right)
\geq (2-d) \left(1- \beta\right).
\end{align*}
Using Cauchy--Schwarz
\begin{align*}
\left(\lim_{} {\rm sup} \frac{\sum_{p \in S_\beta}\frac{|a_p|^4}{p^s}}{\log \left(\frac{1}{s-1}\right)}\right) ^2 
\leq 
\left(\lim_{} {\rm sup} \frac{\sum_{p \in S_\beta}\frac{|a_p|^8}{p^s}}{\log \left(\frac{1}{s-1}\right)}\right)
\left(\lim_{} {\rm sup} \frac{\sum_{p \in S_\beta}\frac{|a_p|^0}{p^s}}{\log \left(\frac{1}{s-1}\right)}\right),
\end{align*}
where the third limit supremum can be bounded above by $1/100$ and we obtain
\begin{align}\label{deq1}
(2-d)^2  \left(1- \beta\right)^2  \leq \frac{14}{100}.
\end{align}

We appeal to a result from ~\cite{Wa16}, which we include here as a lemma.
\begin{lemma}\label{lem}
For $A,B,$ and $d$ defined as above, we have 
\begin{align*}
\lim_{}{\rm sup} \frac{\sum_{p \in A}\frac{|a_p|^3}{p^s}}{\log \left(\frac{1}{s-1}\right)} \geq  \frac{d ^{5/4}}{(14- (2-d)^2) ^{1/4}}.
\end{align*}
\end{lemma}

\begin{proof}
A proof of this Lemma essentially arises in~\cite{Wa16}. We include a proof below for the convenience of the reader.

Using Holder's inequality and taking the limit supremum as $s \rightarrow 1^+$,
\begin{align*}
\lim_{s \rightarrow 1 ^+}{\rm sup} \frac{\sum_{p \in A}\frac{|a_p|^3}{p^s}}{\log \left(\frac{1}{s-1}\right)}
&\leq 
\left( \lim_{s \rightarrow 1 ^+}{\rm sup} \frac{\sum_{p \in A}\frac{|a_p|^4}{p^s}}{\log \left(\frac{1}{s-1}\right)} \right) ^{3/4}
\left( \lim_{s \rightarrow 1 ^+}{\rm sup} \frac{\sum_{p \in A}\frac{1}{p^s}}{\log \left(\frac{1}{s-1}\right)} \right)^{1/4} \\
&\leq (2-d)^{3/4} \cdot 1 ^{1/4}.
\end{align*}

Similarly,
\begin{align*}
\lim_{s \rightarrow 1 ^+}{\rm sup} \frac{\sum_{p \in A}\frac{|a_p|^4}{p^s}}{\log \left(\frac{1}{s-1}\right)}
&\leq 
\left( \lim_{s \rightarrow 1 ^+}{\rm sup} \frac{\sum_{p \in A}\frac{|a_p|^8}{p^s}}{\log \left(\frac{1}{s-1}\right)} \right)^{1/5}
\left( \lim_{s \rightarrow 1 ^+}{\rm sup} \frac{\sum_{p \in A}\frac{|a_p|^3}{p^s}}{\log \left(\frac{1}{s-1}\right)} \right)^{4/5}\\
\Rightarrow \quad \quad \quad  2-d 
&\leq 
\left( \lim_{s \rightarrow 1 ^+}{\rm sup} \frac{\sum_{p \in A}\frac{|a_p|^8}{p^s}}{\log \left(\frac{1}{s-1}\right)} \right)^{1/5}
(2-d)^{3/5} \\
\Rightarrow \quad \quad (2-d)^2 
&\leq
\lim_{s \rightarrow 1 ^+}{\rm sup} \frac{\sum_{p \in A}\frac{|a_p|^8}{p^s}}{\log \left(\frac{1}{s-1}\right)},
\end{align*}
From the results in the previous section, we have 
\begin{align*}
\lim_{s \rightarrow 1 ^+}{\rm sup} \frac{\sum_{p \in A}\frac{|a_p|^8}{p^s}}{\log \left(\frac{1}{s-1}\right)} +
\lim_{s \rightarrow 1 ^+}{\rm inf} \frac{\sum_{p \in B}\frac{|a_p|^8}{p^s}}{\log \left(\frac{1}{s-1}\right)}
\leq 14,
\end{align*}
and so
\begin{align}\label{neweq}
\lim_{s \rightarrow 1 ^+}{\rm inf} \frac{\sum_{p \in B}\frac{|a_p|^8}{p^s}}{\log \left(\frac{1}{s-1}\right)} \leq 14 - (2-d)^2.
\end{align}
We also have 
\begin{align*}
\sum_{p \in B}\frac{|a_p|^{8/5}|a_p|^{12/5}}{p^s} \leq \left(\sum_{p \in B}\frac{|a_p|^8}{p^s}\right)^{1/5}
\left(\sum_{p \in B}\frac{|a_p|^3}{p^s}\right)^{4/5}.
\end{align*}
We divide the equation above by $\log (1 / (s-1))$ and take the limit infimum as $s \rightarrow 1^+$,
\begin{align*}
\lim_{s \rightarrow 1 ^+}{\rm inf} \frac{\sum_{p \in B}\frac{|a_p|^4}{p^s}}{\log \left(\frac{1}{s-1}\right)}
\leq 
\left( \lim_{s \rightarrow 1 ^+}{\rm inf} \frac{\sum_{p \in B}\frac{|a_p|^8}{p^s}}{\log \left(\frac{1}{s-1}\right)} \right)^{1/5}
\left( \lim_{s \rightarrow 1 ^+}{\rm sup} \frac{\sum_{p \in B}\frac{|a_p|^3}{p^s}}{\log \left(\frac{1}{s-1}\right)} \right)^{4/5}.
\end{align*}
We apply equation~\ref{neweq} to obtain
\begin{align*}
d &\leq (14 - (2-d)^2) ^{1/5} \left(\lim_{} {\rm sup} \frac{\sum_{p \in B}\frac{|a_p|^3}{p^s}}{\log \left(\frac{1}{s-1}\right)}\right)^{4/5}\\
\Rightarrow \quad  \frac{d ^{5/4}}{(14- (2-d)^2) ^{1/4}} &\leq \lim_{}{\rm sup} \frac{\sum_{p \in B}\frac{|a_p|^3}{p^s}}{\log \left(\frac{1}{s-1}\right)}.
\end{align*}
For $s > 1$,  we have
\begin{align*}
 \sum_{p \in A}\frac{|a_p|^3}{p^s} + \left(-\sum_{p}\frac{a_p^3}{p^s}\right) =\sum_{p \in B}\frac{|a_p|^3}{p^s},
\end{align*}
and so
 \begin{align*}
 \lim_{s \rightarrow 1 ^+}{\rm sup} \frac{\sum_{p \in A}\frac{|a_p|^3}{p^s}}{\log \left(\frac{1}{s-1}\right)} \geq  \lim_{s \rightarrow 1 ^+}{\rm sup} \frac{\sum_{p \in B}\frac{|a_p|^3}{p^s}}{\log \left(\frac{1}{s-1}\right)}.
 \end{align*}
Therefore 
\begin{align*}
\lim_{s \rightarrow 1 ^+}{\rm sup} \frac{\sum_{p \in A}\frac{|a_p|^3}{p^s}}{\log \left(\frac{1}{s-1}\right)}
\geq \frac{d ^{5/4}}{(14 - (2-d)^2) ^{1/4}}.
\end{align*}
\end{proof}

Now we define $T_\alpha \subset A$ to be the set of primes $p$ such that 
\begin{align*}
|a_p|^3 \geq  \left(\frac{d ^{5/4}}{(14-(2-d)^2)^{1/4}}\right)\alpha,
\end{align*}
where $0 < \alpha < 1$ is a constant to be fixed later.

Let us assume that the upper Dirichlet density of $T_\alpha$ is less than $1/100$. We have 
\begin{align*}
\lim_{} {\rm sup} \left(\frac{\sum_{p \in A-T_\alpha}\frac{|a_p|^3}{p^s}}{\log \left(\frac{1}{s-1}\right)}\right) 
\leq \left(\frac{d ^{5/4}}{(14-(2-d)^2)^{1/4}}\right) \alpha .
\end{align*}

Lemma ~\ref{lem} implies that
\begin{align*}
\lim_{} {\rm sup} \left(\frac{\sum_{p \in A-T_\alpha}\frac{|a_p|^3}{p^s}}{\log \left(\frac{1}{s-1}\right)}\right) 
+ 
\lim_{} {\rm sup} \left(\frac{\sum_{p \in T_\alpha}\frac{|a_p|^3}{p^s}}{\log \left(\frac{1}{s-1}\right)}\right) 
\geq 
\frac{d ^{5/4}}{(14-(2-d)^2)^{1/4}}
\end{align*}
and therefore 
\begin{align*}
\lim_{} {\rm sup} \left(\frac{\sum_{p \in T_\alpha}\frac{|a_p|^3}{p^s}}{\log \left(\frac{1}{s-1}\right)}\right) 
\geq
\frac{d ^{5/4}}{(14-(2-d)^2)^{1/4}} \left(1- \alpha\right).
\end{align*}

We have  
\begin{align*}
\left(\lim_{} {\rm sup} \frac{\sum_{p \in T_\alpha}\frac{|a_p|^3}{p^s}}{\log \left(\frac{1}{s-1}\right)}\right) ^2 
\leq
\left(\lim_{} {\rm sup} \frac{\sum_{p \in T_\alpha}\frac{|a_p|^6}{p^s}}{\log \left(\frac{1}{s-1}\right)}\right) 
\left(\lim_{} {\rm sup} \frac{\sum_{p \in T_\alpha}\frac{|a_p|^0}{p^s}}{\log \left(\frac{1}{s-1}\right)}\right) .
\end{align*}
The second limit supremum can be bounded from above by 5 and the third limit supremum can be bounded from above by $1/100$, so
\begin{align}\label{deq2}
\left(\frac{d ^{5/4}}{(14-(2-d)^2)^{1/4}}\right)^2  \left(1- \alpha\right)^2   \leq \frac{5}{100}.
\end{align}

Given some value for the constant $\beta$, we want to fix $\alpha$ such that 
\begin{align}\label{alpha-equation}
((2-d)\beta)^{1/4} &= \left(\frac{d ^{5/4}}{(14-(2-d)^2)^{1/4}}\alpha\right)^{1/3}.
\end{align}

Given equation~\ref{alpha-equation}, if we set $\beta = 0.495$, we have that if $d \leq 1.258$, then equation~\ref{deq1} is false, contradicting the assumption that $S_\beta$ has an upper Dirichlet density smaller than $1/100$, and if $d > 1.258$, then equation~\ref{deq2} is false, and so $T_\alpha$ would have an upper Dirichlet density of at least $1/100$.
Either way, since for $\beta = 0.495$ and $d = 1.258$ the value of equation ~\ref{alpha-equation}  is $0.36729 ^{1/4} =0.778...$, this implies that the set of primes
\begin{align*}
\{p \mid a_p (\pi) > 0.778... - \epsilon \}.
\end{align*}
has an upper Dirichlet density at least $1/100$.

\begin{remark}
We determined our choice of $\beta$ by solving the following simultaneous equations
\begin{align*}
(2-d)  \left(1- \beta\right)  =& \frac{\sqrt{14}}{10}\\
\frac{d ^{5/4}}{(14-(2-d)^2)^{1/4}} \left(1- \alpha\right)   =& \frac{\sqrt{5}}{10},
\end{align*}
along with equation~\ref{alpha-equation}, and we obtained $\beta = 0.4957...$ and $d = 1.2581...$ .\\
\end{remark}

\subsection*{Acknowledgements}
This work began at the University of Z\"urich, where the author was supported by Forschungskredit grant K-71116-01-01 of the University of Z\"urich and partially supported by grant SNF PP00P2-138906 of the Swiss National Foundation.

\end{document}